\documentclass[12pt]{amsart}
\usepackage[top=30truemm,bottom=30truemm,left=25truemm,right=25truemm]{geometry}
\usepackage{txfonts}
\usepackage{mathrsfs}

\usepackage{bm}
\usepackage{amsfonts,amssymb}
\usepackage{dsfont}
\usepackage{extarrows}
\usepackage{amsmath}
\usepackage{mathrsfs}
\usepackage{enumerate}
\usepackage{amscd}
\usepackage[all]{xy}
\usepackage{hyperref}
\theoremstyle{plain} %text of this environment is typesetted in italics

\theoremstyle{definition} %text of this environment is typesetted in roman letters

\newtheorem{thm}{Theorem}[section]

\newtheorem{lem}[thm]{Lemma}

\theoremstyle{definition}
\newtheorem{defn}{Definition}[section]

\theoremstyle{remark}
\newtheorem{rem}{Remark}[section]

\newcommand{\be}{\begin{equation}}
\newcommand{\ee}{\end{equation}}
\newcommand{\bea}{\begin{eqnarray}}
\newcommand{\eea}{\end{eqnarray}}
\newcommand{\ben}{\begin{eqnarray*}}
	\newcommand{\een}{\end{eqnarray*}}
\newcommand{\bt}{\begin{split}}
	\newcommand{\et}{\end{split}}
\newcommand{\bet}{\begin{equation}}
\newcommand{\mc}{\mathbb{C}}
\newcommand{\mr}{\mathbb{R}}
\newcommand{\ra}{\rightarrow}
\newcommand{\beq}{\begin{equation*}}
\newcommand{\eeq}{\end{equation*}}
\begin{document}

\title[Curvature positivity of invariant direct images ]
{Curvature positivity of invariant direct images of Hermitian vector bundles}

\author[F. Deng]{Fusheng Deng}
\address{Fusheng Deng: \ School of Mathematical Sciences, University of Chinese Academy of Sciences\\ Beijing 100049, P. R. China}
\email{fshdeng@ucas.ac.cn}

\author[J. Hu]{Jinjin Hu}
\address{Jinjin Hu: \ School of Mathematical Sciences, University of Chinese Academy of Sciences\\ Beijing 100049, P. R. China}
\email{hujinjin18@mails.ucas.ac.cn}

\author[W. Jiang]{Weiwen Jiang}
\address{Weiwen Jiang: \ School of Mathematical Sciences, University of Chinese Academy of Sciences\\ Beijing 100049, P. R. China}
\email{jiangweiwen17@mails.ucas.edu.cn}

\begin{abstract}
We prove that the invariant part, with respect to a compact group action satisfying certain condition,
of the direct image of a Nakano positive Hermitian holomorphic vector bundle over a bounded pseudoconvex domain
is Nakano positive. We also consider the action of the noncompact group $\mr^m$ and get the same result for a family of tube domains,
which leads to a new method to the
matrix-valued Prekopa's theorem originally proved by Raufi.
The two main ingredients in our method are H\"ormander's $L^2$ theory of $\bar\partial$ and the recent work of Deng-Ning-Zhang-Zhou on characterization of Nakano positivity
of Hermitian holomorphic vector bundles.
\end{abstract}
%\thanks{(*) The second author and the third author are both corresponding authors.}

\maketitle

%%%%%%%%%%%%%%%%%%%%%%%%%%%%%%%%%%%%%%%%%%%%%%%%%%%%%%
\section{Introduction}\label{sec:intro}
Let $p:\mc^n_t\times\mc^m_z\ra\mc^n$ be the natural projection.
For a domain $\Omega\subset\mc^n\times\mc^m$, we denote the fiber $p^{-1}(t)\cap\Omega$ of $\Omega$ over $t$ by $\Omega_t$ for $t\in p(\Omega)\subset\mc^n_t$.
In 1998, Berndtsson proved the following remarkable result which stimulates a series of important works on positivity of direct image sheaves of positively
curved Hermitian holomorphic vector bundles.

\begin{thm}[\cite{Ber98}]\label{thm:Berndtsson minimum}
Let $\varphi(t,z)$ be a plurisubharmonic function on a pseudoconvex domain $\Omega\subset \ \mathbb{C}_{t}^{n}\ \times \ \mathbb{C}_{z}^{m}$.
\begin{itemize}
\item[(1)] If all fibers $\Omega_t\ (t\in p(\Omega))$ are (connected) Reinhardt domains, and $\varphi(t,z)$ is independent of  $Arg(z_{j}),\ j=1,\cdots,m$,
then the function $\tilde{\varphi}$ defined by
$$e^{-\tilde{\varphi}(t)}\ =\ \int_{\Omega_t} e^{-\varphi(t,z)}dV_z$$
is a p.s.h function on $p(\Omega)$, where $dV_z$ is the Lebsgue measure on $\mathbb C^m$.
Moreover, if all $\Omega_t$ contain the origin, we only assume $\Omega_t$ and $\varphi$ are invariant under the
transform $z\ \longmapsto\ e^{i\theta}z,$ $\forall \theta \in \mathbb{R}$, then the same result still holds.
\item[(2)] If all fibers $\Omega_t$ are tube domains,
 $$\Omega_t =X_t + i\mathbb{R}^{m}$$
 and $\varphi$ is independent of  $Im(z_{j}),\ j=1,\cdots,m$,
 then the function $\tilde{\varphi}$ defined by
 $$e^{-\tilde{\varphi}(t)}\ =\ \int_{X_t} e^{-\varphi(t,Re(z))}dV_{Re(z)}$$
 is a p.s.h function on $p(\Omega)$.
\end{itemize}
\end{thm}

The above result is motivated by and generalizes Kiselmans's minimum principle for plurisubharmonic functions \cite{Kis78} and Prekopa's theorem for convex functions \cite{Pre73}.
Kiselmans's minimum principle states that, under the condition of (2) in Theorem \ref{thm:Berndtsson minimum},
the function
$$\varphi^*(t)=\inf_{z\in\Omega_t}\varphi(t,z)$$
is a plurisubharmonic function on $p(\Omega)$,
and Prekopa's theorem states that for a convex function $\phi(x,y)$ on $\mr^n_x\times\mr^m_y$, the function $\tilde\phi(x)$ defined by
$$e^{-\tilde\phi(x)}=\int_{\mr^m}e^{-\phi(x,y)}dV_y$$
is a convex function on $\mr^n_x$.

In \cite{Ber06,Ber09,DZZ14,DZZ17}, Theorem \ref{thm:Berndtsson minimum} was generalized along different directions.

Motivated by Raufi's work on matrix valued Prekopa's theorem \cite{Rau13}
and the recent work of the first author and collaborators on the characterization of Nakano positivity
of Hermitian holomorphic vector bundles\cite{DNW19,DNWZ19}, we generalize (1) in Theorem \ref{thm:Berndtsson minimum} to the following:

\begin{thm}\label{thm:intr-minimum compact action}
Let $\Omega \subseteq \mc_{t}^{n} \times \mc_{z}^{m}$ be a pseudoconvex domain,
such that $\Omega_t$ are connected for all $t$ in $D:=p(\Omega)$.
Let $\tilde E=\Omega\times\mc^r$ be the trivial holomorphic vector bundle of rank $r$ over $\Omega$,
and let $\tilde h(t,z)$ be an Hermitian metric on $\tilde E$ which is viewed as a smooth map from $\Omega$ to the space of positive Hermitian matrices.
Let $E=D\times \mc^r$ be the trivial bundle over $D$ with the hermitian metric given by
$$h(t)=\int_{{p}^{-1}(t)}\tilde h(t,z)dV_z.$$
Assume that there is a compact Lie group $K$ acting holomorphically on $\Omega$ by acting on the second variable $z$, such that
\begin{itemize}
\item[(i)] $\tilde{h}(t, z) dV_{z}$ is $K$-invariant for $t\in D$, and
\item[(ii)] all $K$-invariant holomorphic functions on $\Omega_t$ are constant, $t\in D$.
\end{itemize}
If $(\tilde E,\tilde h)$ is Nakano semi-positive and $h$ is $C^2$, then $(E, h)$ is Nakano semi-positive.
\end{thm}
\begin{rem}
In Theorem \ref{thm:intr-minimum compact action}, if $K$ is not compact as assumption, then $h$ would be identically equal to $+\infty$.
\end{rem}

Our method to Theorem \ref{thm:intr-minimum compact action} is different from those in \cite{Ber98} and \cite{Ber09}.
The idea of the proof is as follows.
Since $(\tilde E,\tilde h)$ is Nakano semi-positive,
it satisfies the optimal $L^2$-estimate condition (see \S \ref{sec:pre} for definition) by H\"ormander's $L^2$-estimate of $\bar\partial$,
which can imply that $(E, h)$ also satisfies the optimal $L^2$-estimate condition.
By \cite[Theorem 1.1]{DNWZ19}, we see that $(E, h)$ is Nakano semi-positive.

With the same method, Theorem \ref{thm:intr-minimum compact action} can be generalized to general K\"ahler fibrations and Nakano
semi-positive Hermitian vector bundles. But we just restrict ourselves to the basic context as above.
If the fibers $\Omega_t$ in Theorem \ref{thm:intr-minimum compact action} are not assumed to be connected,
we can also get a similar result by considering the GIT quotient $\Omega//K$, as in \cite{DZZ14}.
Under the assumption as in Theorem \ref{thm:intr-minimum compact action}, 
we can canonically identify $\Omega//K$ with $D$.

Theorem \ref{thm:intr-minimum compact action} considers actions of compact groups.
However, following the idea in \cite{DZZ14,DZZ17}, it can be generalized to certain noncompact group actions.
In the present context, we consider tube domains as in Theorem \ref{thm:Berndtsson minimum} which corresponds to the action of the group $\mr^m$.

\begin{thm}\label{thm:intro-minimum noncompact action}
Let $\Omega \subseteq \mc_{t}^{n} \times \mc_{z}^{m}$ be a pseudoconvex domain,
such that $\Omega_t=U_t\times i\mr^m$ are (connected) tube domains for all $t$ in $D:=p(\Omega)$.
Let $\tilde E=\Omega \times\mc^r\ra\Omega$
be the trivial holomorphic vector bundle of rank $r$ on $\Omega$.
Let $\tilde h(t,z)$ be an Hermitian metric on $\tilde E$, which is independent of the imaginary part $Im z$ of $z$.
Let $h(t)$ be the Hermitian metric on the trivial vector bundle $E=D \times\mc^r\ra D$ over $D$,
given by
$$h(t):=\int_{U_t}\tilde h(t,Rez)d V_{Rez}.$$
If ($\tilde E, \tilde h$) is Nakano semi-positive and $h$ is $C^2$,
then ($E, h$) is Nakano semi-positive.
\end{thm}

A Corollary of Theorem \ref{thm:intro-minimum noncompact action} is the following

\begin{thm}\label{thm:intro real convex case}
Let $\Omega_0 \subseteq \mr_{t}^{n} \times \mr_{x}^{m}$ be a convex domain,
let $p_0:\Omega_0 \rightarrow \mr_{t}^{n}$ be the natural projection,
and let $\Omega_{0,t}=p_0^{-1}(t)$ for $t\in D_0:=p_0(\Omega_0)$.
Let $\tilde g(t,x):\Omega_0 \rightarrow GL(r,\mc)$ be an Hermitian metric on the trivial vector bundle
$\tilde E_0=\Omega_0 \times\mc^r\ra \Omega_0$.
Let $g(t):D_0 \ra GL(r,\mc)$ be the Hermitian metric on the trivial vector bundle $E_0=D_0 \times\mc^r\ra D_0$ over $D_0$
given by
$$g(t):=\int_{\Omega_{0,t}}\tilde g(t,x)dV_x.$$
If ($\tilde E_0, \tilde g$) is Nakano semi-positive and $g$ is $C^2$,
then ($E_0, g$) is Nakano semi-positive.
\end{thm}

The concept of Nakano positivity of $(\tilde E_0, \tilde g_0)$ was introduced in \cite{Rau13}(phrased as Nakano log concave there) 
and will be recalled in \S \ref{sec:pre}.

In the case that $\Omega=\mr^n\times\mr^m$, Theorem \ref{thm:intro real convex case}
is proved and called the  matrix-valued Prekopa's theorem by Raufi in \cite{Rau13}.
The proof in \cite{Rau13} contains two main ingredients-
Berndtsson's method to the positivity of direct image bundles \cite{Ber09} and a Fourier transform technique.
To avoid the Fourier transform technique and complex analysis in the proof,
Cordero-Erausquin recently produced a new proof of the matrix-valued Prekopa's theorem based on $L^2$-methods in real analysis \cite{Cor19}.

Our method to Theorem \ref{thm:intro real convex case} is different from those in \cite{Rau13} and \cite{Cor19}.
With Theorem \ref{thm:intr-minimum compact action} at hand, we reduce Theorem \ref{thm:intro-minimum noncompact action}
to Theorem \ref{thm:intr-minimum compact action} by considering the covering map $\pi:\mc^m\ra (\mc^*)^n$.
This idea is motivated by the work in \cite{DZZ14}.
Then it is obvious that Theorem \ref{thm:intro real convex case} is a  consequence of Theorem \ref{thm:intro-minimum noncompact action}.
In this way, we avoid involving vector bundles of infinite rank and the Fourier transform technique.

$\mathbf{Acknowlegements.}$
The authors are grateful to Professor Jiafu Ning, Zhiwei Wang, and Xiangyu Zhou for helpful discussions.
The authors are partially supported by the NSFC grant 11871451.

\section{Preliminaries}\label{sec:pre}
We recall some notions and known results that will be used later.

\begin{defn}[\cite{DNWZ19}]\label{def:optimal L2 estimate cond}
Let $\Omega\subset\mc^{n}$ be a bounded domain.
Let $E=\Omega\times\mc^r$ be the trivial holomorphic vector bundle over $\Omega$ whose canonical frame is denoted by $\{e_1,\cdots, e_r\}$.
Let $h$ be a Hermitian metric on $E$ such that $h(e_\lambda, e_\mu)=h_{\lambda\mu}$.
We say that $(E,h)$ satisfies \emph{the optimal $L^2$-estimate condition},
if for any smooth strictly plurisubharmonic function $\psi$ on $\Omega$, for any $f \in C_{c}^{\infty}\left(\Omega, \wedge^{0, 1} \mathrm{T}^{*}\Omega \otimes \mathrm{E}\right)$ with $\bar{\partial} f=0$,
there is a locally integrable section $u$ of $E$, satisfying $\bar{\partial} u=f$ in the sense of distribution, and
$$\int_{\Omega}|u|_{h}^{2} e^{-\psi} d V \leqslant \int_{\Omega} \sum_{i, j,\lambda,\mu} \psi^{\text {ij}} f_{i \lambda} \overline{f_{j \mu}}h_{\lambda \mu} e^{-\psi}dV,$$
where $dV$ is the Lebesgue measure,
$(\psi^{i j})=\left(\psi_{i j}\right)^{-1}$=$\left(\frac{\partial^{2} \psi}{\partial z_{i} \partial \bar{z}_{j}}\right)^{-1}$, and $f=\sum_{\lambda=1}^r(\sum_{j=1}^{n} f_{j \lambda} d \bar{z}_{j}) \otimes e_{\lambda}$.
\end{defn}

By the $L^2$-estimate of $\bar\partial$ by H\"ormander and Demailly, we have

\begin{lem}[c.f. {\cite[Theorem 4.5]{Dem}}]\label{lem:L^2 estimate}
Let $\Omega, E, h$ be as in Definition \ref{def:optimal L2 estimate cond}.
If $\Omega$ is pseudoconvex and $(E,h)$ is Nakano semi-positive, then
$(E,h)$ satisfies the optimal $L^2$-estimate condition.
\end{lem}

For the definition of curvature and Nakano positivity for Hermitian holomorphic vector bundles, see \cite{Dem}.
Recently, the converse of Lemma \ref{lem:L^2 estimate} was established.

\begin{lem}[{\cite[Theorem 1.1]{DNWZ19}}]\label{lem:Nakano cha}
If $(E,h)$ satisfies the optimal $L^2$-estimate condition,then $(E,h)$ is Nakano semi-positive.
\end{lem}

\begin{defn}[{\cite[Definition 2]{Rau13}}]\label{def:nakano positive}
Let $g:\Omega\ra GL(r,\mc)$ be an Hermitian metric on the trivial complex vector bundle $E=\Omega\times\mc^r$ over an open set $\Omega\subset\mr^n$.
Let $$\Theta^g_{jk}=-\frac{\partial}{\partial x_k}\left(g^{-1}\frac{\partial g}{\partial x_j}\right),\ 1\leq j,k\leq n,$$
where differentiation should be interpreted elementwise.
We say that $(E,g)$ is Nakano semi-positive if for any $n$-tuple of vectors $\{u_j\}^r_{j=1}\subset\mc^r$
$$\sum^n_{j,k=1}(\Theta^g_{jk}u_j, u_k)_g\geq 0.$$
\end{defn}

\begin{rem}\label{rem:real and complex Nak pos}
Let $U\subset\mr^n$ be a connected open set and let $\Omega=U+i\mr^n\subset \mc^n$ be the tube domain with base $U$.
Let $h(z):U\ra GL(r,\mc)$ be an Hermitian metric on the trivial holomorphic vector bundle $E=\Omega\times\mc^r$.
Assume that $h(z)$ is independent of the imaginary part of $z$.
By definition, one can see that $(E,h)$ is Nakano semipositive as an Hermitian holomorphic vector bundle 
if and only if $(E|_U=U\times\mc^r, h|_U)$ is Nakano semipositive in the sense of Definition \ref{def:nakano positive}.
\end{rem}

\section{Positivity of invariant direct images with compact group actions}
The aim of this section is to prove Theorem \ref{thm:intr-minimum compact action}.
For convenience, we restate it here.
\begin{thm}[=Theorem \ref{thm:intr-minimum compact action}]\label{thm:minimum compact action}
Let $\Omega \subseteq \mc_{t}^{n} \times \mc_{z}^{m}$ be a pseudoconvex domain,
such that $\Omega_t$ are connected for all $t$ in $D:=p(\Omega)$.
Let $\tilde E=\Omega\times\mc^r$ be the trivial holomorphic vector bundle of rank $r$ over $\Omega$,
and let $\tilde h(t,z)$ be an Hermitian metric on $\tilde E$ which is viewed as a smooth map from $\Omega$ to the space of positive Hermitian matrices.
Let $E=D\times \mc^r$ be the trivial bundle over $D$ with the hermitian metric given by
$$h(t)=\int_{{p}^{-1}(t)}\tilde h(t,z)dV_z.$$
Assume that there is a compact Lie group $K$ acting holomorphically on $\Omega$ by acting on the second variable $z$, such that
\begin{itemize}
\item[(i)] $\tilde{h}(t, z) dV_{z}$ is $K$-invariant for $t\in D$, and
\item[(ii)] all $K$-invariant holomorphic functions on $\Omega_t$ are constant, $t\in D$.
\end{itemize}
If $(\tilde E,\tilde h)$ is Nakano semi-positive and $h$ is $C^2$, then $(E, h)$ is Nakano semi-positive.
\end{thm}
\begin{proof}
By Lemma \ref{lem:Nakano cha}, it suffices to prove that $(E, h)$ satisfies the optimal $L^2$ estimate condition.

Assume that $(e_{1},...,e_{r})$ is the canonical holomorphic frame of $E$,
and $(\tilde{e}_{1},...,\tilde{e}_{r})$ is the canonical holomorphic frame of $\tilde{E}$.
Let $\psi$ be a smooth strictly plurisubharmonic function on $D$.
Setting $\tilde{\psi}(t, z)=\psi(t), \ (t,z)\in \Omega$, we get a smooth plurisubharmonic function $\tilde\psi$ on $\Omega$.

Let $f \in C^{\infty}\left(D, \wedge^{0, 1} \mathrm{T}^{*}D \otimes \mathrm{E}\right)$
with compact support and $\bar\partial f=0$.
We can write $f$ as
$$f=\sum_{\lambda=1}^r w_{\lambda} \otimes e_{\lambda}=\sum_{\lambda=1}^r\left(\sum_{j=1}^{n} f_{j \lambda} d \bar{t}_{j}\right) \otimes e_{\lambda},$$
where $f_{j\lambda}$ are smooth functions on $D$.
Let $\tilde w_\lambda=p^*(w_\lambda)$, $\tilde f_{j\lambda}=p^*(f_{j\lambda}):=f_{j\lambda}\circ p$,
and let
$$\tilde{f}=\sum_{\lambda=1}^r \tilde w_{\lambda} \otimes \tilde{e}_{\lambda}=\sum_{\lambda=1}^r\left(\sum_{j=1}^{n} \tilde f_{j \lambda} d \bar{t}_{j}\right) \otimes \tilde{e}_{\lambda} \in C^{\infty}\left(\Omega, \wedge^{0, 1} \mathrm{T}^{*}\Omega \otimes \tilde{E}\right),$$
then $\bar\partial\tilde f=0$.

Let $\Theta_{\tilde E,\tilde h}$ be the curvature operator of $(\tilde E, \tilde h)$.
Then $\Theta+\partial\bar\partial\tilde\psi$ is the curvature operator of $\tilde E$ with the Hermitian metric $e^{-\tilde\psi}\cdot\tilde h$.
Since $(\tilde E,\tilde h)$ is Nakano semi-positive, by Lemma \ref{lem:L^2 estimate},
there exists $\tilde u\in L^2(\Omega,\tilde E)$ such that
\begin{equation}\label{eqn:dbar total space}
\bar\partial\tilde u=\tilde f,
\end{equation}
and
\begin{equation}\label{eqn:L2 direct image}
\begin{split}
\int_\Omega |\tilde u|^2_{\tilde h}e^{-\tilde\psi}dV_{(t,z)}
&\leq \int_{\Omega}\langle[i\Theta_{\tilde E,\tilde h}+i\partial\bar\partial\tilde\psi, \Lambda]^{-1}\tilde f, \tilde f\rangle e^{-\tilde\psi}dV_{(t,z)}\\
&\leq \int_{\Omega}\langle[i\partial\bar\partial\tilde\psi, \Lambda]^{-1}\tilde f, \tilde f\rangle e^{-\psi}dV_{(t,z)}\\
&= \int_{\Omega}\sum_{i, j=1}^n\sum_{\lambda,\mu=1}^r \tilde{\psi}^{\text {ij}} \tilde f_{i \lambda} \bar{\tilde f}_{j \mu} \tilde{h}_{\lambda \mu}e^{-\tilde{\psi}}dV_{(t,z)},\\
\end{split}
\end{equation}
where $\Lambda$ is the adjoint of the operator given by the wedge product of the flat K\"ahler form on $\Omega$,
and $(\tilde\psi^{ij})_{n\times n}$ is the inverse of the matrix $(\frac{\partial ^2\tilde \psi}{\partial t_i\partial\bar t_j})_{n\times n}$.

We assume that $\tilde u$ is minimal, in the sense that
the left hand side in the top line in \eqref{eqn:L2 direct image} is minimal.
We write $\tilde u$ as
$$\tilde u=\sum^r_{\lambda=1}\tilde u_\lambda \tilde e_\lambda,$$
where $\tilde u_\lambda$ are functions on $\Omega$.
From equation \eqref{eqn:dbar total space}, we have $\bar\partial\tilde u_\lambda=\tilde w_\lambda$,
which means that $\bar\partial_z \tilde u_\lambda=0$ and hence $\tilde u_\lambda$ are holomorphic with $z_1,\cdots, z_m$.

For any $g\in K$, let
$$\tilde u_g=\sum^r_{\lambda=1}\tilde u_\lambda(t,gz) \tilde e_\lambda,$$
then it is clear that $\tilde u_g$ also satisfies the equation $\bar\partial\tilde u_g=\tilde f$.
By assumption, we also have
$$\int_\Omega |\tilde u_g|^2_{\tilde h}e^{-\tilde\psi}dV_{(t,z)}=\int_\Omega |\tilde u|^2_{\tilde h}e^{-\tilde\psi}dV_{(t,z)}.$$
By the uniqueness of the minimal solution, we have $\tilde u_g=\tilde u$ and hence $\tilde u_\lambda(t,gz)=\tilde u_\lambda(t,z)$ for all $g\in K$.
By assumption, $\tilde u_\lambda(t,z)$ must be independent of $z$.
So we can view $\tilde u_\lambda(t,z)$ as a function on $D$, denoted by $u_\lambda(t)$.

Let $u=\sum_{\lambda=1}^r u_{\lambda} e_{\lambda} \in L^{2}\left(D, E\right)$,
then it is clear that $\bar{\partial} u=f$.
By Fubini's theorem, we get
$$\int_{\Omega}|\tilde{u}|_{\tilde{h}}^{2} e^{-\tilde{\psi}} dV_{(t,z)}=\int_{\Omega}\tilde u_{\lambda} \bar{\tilde u}_{\mu} \tilde{h}_{\lambda \mu} e^{-\tilde{\psi}} d V_{(t,z)}=\int_{D}u_{\lambda} \bar{u}_{\mu} h_{\lambda \mu} e^{-\psi} dV_{t}=\int_{D}|u|_{h}^{2} e^{-\psi} d V_{t},$$
$$\int_{\Omega} \sum_{i, j=1}^n\sum_{\lambda,\mu=1}^r \tilde{\psi}^{ij}\tilde f_{i \lambda} \bar{\tilde f}_{j \mu}\tilde{h}_{\lambda \mu} e^{-\tilde{\psi}}d V_{(t,z)}=\int_{D} \sum_{i, j=1}^n\sum_{\lambda,\mu=1}^r \psi^{ij} f_{i \lambda} \bar f_{j \mu}h_{\lambda \mu} e^{-\psi}d V_{t}.$$
Combing the above identities with estimate \eqref{eqn:L2 direct image}, we get
$$\int_{D}|u|_{h}^{2} e^{-\psi} d V_{t} \leqslant \int_{D} \sum_{i, j=1}^n\sum_{\lambda,\mu=1}^r \psi^{ij} f_{i \lambda} \bar f_{j \mu}h_{\lambda \mu} e^{-\psi}d V_{t},$$
which implies that $(E,h)$ satisfies the optimal $L^2$ condition.
By Lemma \ref{lem:Nakano cha},  $(E,h)$ is Nakano semi-positive.
\end{proof}

\section{The case of tube domains}\label{sec:tube domain case}
We give the proof of Theorem \ref{thm:intro-minimum noncompact action} and Theorem \ref{thm:intro real convex case}
\begin{thm}[=Theorem \ref{thm:intro-minimum noncompact action}]\label{thm:minimum noncompact action}
Let $\Omega \subseteq \mc_{t}^{n} \times \mc_{z}^{m}$ be a pseudoconvex domain,
such that $\Omega_t=U_t\times i\mr^m$ are (connected) tube domains for all $t$ in $D:=p(\Omega)$.
Let $\tilde E=\Omega \times\mc^r\ra\Omega$
be the trivial holomorphic vector bundle of rank $r$ on $\Omega$.
Let $\tilde h(t,z)$ be an Hermitian metric on $\tilde E$, which is independent of the imaginary part $Im z$ of $z$.
Let $h(t)$ be the Hermitian metric on the trivial vector bundle $E=D \times\mc^r\ra D$ over $D$,
given by
$$h(t):=\int_{U_t}\tilde h(t,Rez)d V_{Rez}.$$
If ($\tilde E, \tilde h$) is Nakano semi-positive and $h$ is smooth,
then ($E, h$) is Nakano semi-positive.
\end{thm}

\begin{proof}
Let us consider the following map:
\begin{equation}
\begin{split}
f: \Omega &\rightarrow \Omega^{*}_{(t,w)}\\
(t_1,\cdots,t_n, z_1,\cdots, z_m)& \mapsto (t_1,\cdots, t_n, e^{z_1}, \cdots, e^{z_m}),
\end{split}
\end{equation}
where $\Omega^{*}=f(\Omega)$.
Since $\tilde h(t,z)$ is independent of the imaginary part of $z$,
it induces a metric $\tilde h'(t,w): \Omega^{*} \ra GL(r,\mc)$
on the trivial bundle
$$\tilde E':=\Omega^{*} \times \mc^r\ra \Omega^{*},$$
which is given by
$$\tilde h'(t,w)=\tilde h(t,\ln|w|).$$
Then we have
\beq
\begin{split}
h(t)
&=\int_{U_t} \tilde h(t,x)d V_x\\
&=\frac{1}{(2 \pi)^{m}} \int_{\Omega^*_t} \frac{1}{|w_{1}|}\cdots \frac{1}{|w_{m}|} \tilde h^{'}(t, w) d V_w\\
&=\int_{\Omega^*_t} \tilde h''(t,w)dV_w.
\end{split}
\eeq
where $h''(t,w):=\frac{1}{(2 \pi)^{m}}\frac{1}{|w_{1}|}\cdots \frac{1}{|w_{m}|} \tilde h^{'}(t, w)$ can be viewed as a Hermitian metric on $\tilde E'$.

The curvature of $(\tilde E', \tilde h'')$ is given by
$$\Theta_{\tilde E',\tilde h''}=\Theta_{\tilde E',\tilde h'}-\sum^m_{j=1}\partial\bar\partial(\ln|w_j|)-\partial\bar\partial\ln(2\pi)^m=\Theta_{\tilde E',\tilde h'}.$$
So $(\tilde E', \tilde h'')$ is also Nakano semi-positive.

Considering the compact Lie group $K:=(S^1)^m$ which acts on $p^{-1}(t)$ as
$$(\alpha_1,\cdots, \alpha_m)(t,w_1,\cdots, w_m)=(t, \alpha_1w_1, \cdots, \alpha_mw_m),$$
and applying Theorem \ref{thm:minimum compact action}, we see that $(E, h)$ is Nakano semi-positive.
\end{proof}

\begin{thm}[=Theorem \ref{thm:intro real convex case}]
Let $\Omega_0 \subseteq \mr_{t}^{n} \times \mr_{x}^{m}$ be a convex domain,
let $p_0:\Omega_0 \rightarrow \mr_{t}^{n}$ be the natural projection,
and let $\Omega_{0,t}=p_0^{-1}(t)$ for $t\in D_0:=p_0(\Omega_0)$.
Let $\tilde g(t,x):\Omega_0 \rightarrow GL(r,\mc)$ be an Hermitian metric on the trivial vector bundle
$\tilde E_0=\Omega_0 \times\mc^r\ra \Omega_0$.
Let $g(t):D_0 \ra GL(r,\mc)$ be the Hermitian metric on the trivial vector bundle $E_0=D_0 \times\mc^r\ra D_0$ over $D_0$
given by
$$g(t):=\int_{\Omega_{0,t}}\tilde g(t,x)dV_x.$$
If ($\tilde E_0, \tilde g$) is Nakano semi-positive and $g$ is $C^2$,
then ($E_0, g$) is Nakano semi-positive.
\end{thm}

\begin{proof}
Let $\Omega=\Omega_0 +i\mr^{n+m}\subset \mc^n\times\mc^m$.
We extend $\tilde g$ to an Hermitian metric $\tilde h$ on the trivial bundle $\tilde E:=\Omega\times\mc^r$ such that $\tilde h(t,z)$
is independent of the imaginary part of $t, z$ and $\tilde h|_{\Omega_0}=\tilde g$.
Let $p:\mc^n\times\mc^m\ra\mc^n$ be the natural projection and let $D=p(\Omega)\subset\mc^n$ and $\Omega_t=p^{-1}(t)=U_t\times i\mr^m$.
We define an Hermitian metric on the trivial bundle $E=D\times\mc^r$ by setting
$$h(t)=\int_{U_t}\tilde h(t,Rez)dV_{Rez}.$$
By Theorem \ref{thm:minimum noncompact action}, $(E,h)$ is Nakano semi-positive.
It is clear that $h(t)$ is independent of the imaginary part of $t$ and $h|_{D_0}=g$,
thus $(E_0,g)$ is Nakano semi-positive (see Remark \ref{rem:real and complex Nak pos}).
\end{proof}

%
%\bibliographystyle{abbrv}
%\bibliography{/Users/zhiwei/Paper/Bib/RefWang.bib}

\bibliographystyle{amsplain}

\end{document}